\documentclass{article}
\usepackage{amssymb}
\usepackage{amsfonts}
\usepackage{amsmath}

\newtheorem{theorem}{Theorem}

\newtheorem{corollary}[theorem]{Corollary}

\newtheorem{definition}[theorem]{Definition}

\newtheorem{remark}[theorem]{Remark}

\newenvironment{proof}[1][Proof]{\noindent\textbf{#1.} }{\ \rule{0.5em}{0.5em}}

\begin{document}

\begin{center}
{\large \textbf{Inclusion properties for bi-univalent functions of complex
order defined by combining of Faber polynomial expansions and Fibonacci
numbers}}\\[2mm]
\textbf{\c{S}ahsene Alt\i nkaya$^{1,\ast }$, Samaneh G. Hamidi$^{2}$, Jay M.
Jahangiri$^{3}$, Sibel Yal\c{c}\i n$^{1}$} \\[2mm]

$^{1}$Department of Mathematics,\\[0pt]
Bursa Uludag University, 16059 Bursa, Turkey\\[0pt]
\textbf{E-Mail: sahsenealtinkaya@gmail.com, syalcin@uludag.edu.tr}\\[0pt]
$^{2}$Department of Mathematics, Brigham Young University, \\[0pt]
Provo, UT 84602, USA\\[0pt]
\textbf{E-Mail: shamidi@rider.edu}\\[1mm]
$^{3}$Department of Mathematical Sciences, Kent State University, \\[0pt]
Burton, OH 44021-9500, USA\\[0pt]
\textbf{E-Mail: jjahangi@kent.edu}\\[1mm]

\begin{abstract}
In this present investigation, we introduce the new class $\mathfrak{R}%
_{\Sigma ,\gamma }^{\mu ,\rho }\left( \widetilde{\mathfrak{p}}\right) $ of
bi-univalent functions defined by using the Tremblay fractional derivative
operator. Additionally, we use the Faber polynomial expansions and Fibonacci
numbers to derive bounds for the general coefficient $\left\vert
a_{n}\right\vert $ of the bi-univalent function class.


2010, Mathematics Subject Classification: 30C45, 33D15.
\end{abstract}

\end{center}

\section{Introduction, Definitions and Notations}

Let $%
\mathbb{C}
$ be the complex plane and $\mathbb{U}=\left\{ z:z\in 
\mathbb{C}
\text{ and }\left\vert z\right\vert <1\right\} $ be open unit disc in $%
\mathbb{C}
$. Further, let $\mathcal{A}$ represent the class of functions analytic in $%
\mathbb{U}$, satisfying the condition%
\begin{equation*}
f(0)=\ f^{\prime }(0)-1=0.
\end{equation*}%
Then each function $f$ in $\mathcal{A}$ has the following Taylor series
expansion%
\begin{equation}
f(z)=z+a_{2}z^{2}+a_{3}z^{3}+\cdots =z+\overset{\infty }{\underset{n=2}{\sum 
}}a_{n}z^{n}.  \label{eq1}
\end{equation}%
The class of this kind of functions is represented by $\mathcal{S}$.

With a view to reminding the rule of subordination for analytic functions,
let the functions $f,g$ be analytic in $\mathbb{U}$. A function $f$ is 
\textit{subordinate} to $g,$ indited as $f\prec g,$ if there exists a
Schwarz function 
\begin{equation*}
\mathbf{\varpi }(z)=\overset{\infty }{\underset{n=1}{\sum }}\mathfrak{c}%
_{n}z^{n}\ \ \left( \mathbf{\varpi }\left( 0\right) =0,\text{\ }\left\vert 
\mathbf{\varpi }\left( z\right) \right\vert <1\right) ,
\end{equation*}%
analytic in $\mathbb{U}$ such that 
\begin{equation*}
f\left( z\right) =g\left( \mathbf{\varpi }\left( z\right) \right) \ \ \ \
\left( z\in \mathbb{U}\right) .
\end{equation*}%
For the Schwarz function $\mathbf{\varpi }\left( z\right) $ we know that $%
\left\vert \mathfrak{c}_{n}\right\vert <1$ (see \cite{Duren 83}).

According to the \textit{Koebe-One Quarter Theorem}, every univalent
function $f\in \mathcal{A}$ has an inverse $f^{-1}$ satisfying $f^{-1}\left(
f\left( z\right) \right) =z~~\left( z\in \mathbb{U}\right) $ and $f\left(
f^{-1}\left( w\right) \right) =w~$ $\left( \left\vert w\right\vert
<r_{0}\left( f\right) ;~~r_{0}\left( f\right) \geq \frac{1}{4}\right) ,$
where%
\begin{equation}
\begin{array}{l}
g(w)=f^{-1}\left( w\right) =w~-a_{2}w^{2}+\left( 2a_{2}^{2}-a_{3}\right)
w^{3} \\ 
\\ 
\ \ \ \ \ \ \ \ \ \ \ -\left( 5a_{2}^{3}-5a_{2}a_{3}+a_{4}\right)
w^{4}+\cdots .%
\end{array}
\label{eq2}
\end{equation}%
A function $f\in \mathcal{A}$ is said to be bi-univalent in $\mathbb{U}$ if
both $f$ and $f^{-1}$ are univalent in $\mathbb{U}.~$Let $\Sigma $ denote
the class of bi-univalent functions in $\mathbb{U}$ given by (\ref{eq1}).
For a brief historical account and for several notable investigation of
functions in the class $\Sigma ,$ see the pioneering work on this subject by
Srivastava et al. \cite{Srivastava 2010} (see also \cite{Brannan and Clunie
80, Brannan and Taha 86, Lewin 67, Netanyahu 69}). The interest on estimates
for the first two coefficients $\left\vert a_{2}\right\vert $, $\left\vert
a_{3}\right\vert $ of the bi-univalent functions keep on by many researchers
(see, for example, \cite{AA, Hayami 2012, HO, Seker 2016, Srivastava 2013}).
However, in the literature, there are only a few works (by making use of the
Faber polynomial expansions) determining the general coefficient bounds $%
\left\vert a_{n}\right\vert $ for bi-univalent functions (\cite{AY, Hamidi
and Jahangiri 2014, Hamidi and Jahangiri 2016, S}). The coefficient estimate
problem for each of $\left\vert a_{n}\right\vert $ $\left( \ n\in 
\mathbb{N}
\backslash \left\{ 1,2\right\} ;\ \ 
\mathbb{N}
=\left\{ 1,2,3,...\right\} \right) $ is still an open problem.

Now, we recall to a notion of $q$-operators that play a major role in
Geometric Function Theory. The application of the $q$-calculus in the
context of Geometric Function Theory was actually provided and the basic (or 
$q$-) hypergeometric functions were first used in Geometric Function Theory
in a book chapter by Srivastava \cite{Srivastava1989}. For the convenience,
we provide some basic notation details of $q$-calculus which are used in
this paper.

\begin{definition}
(See \cite{SO}) For a function $f$ (analytic in a simply-connected region of 
$%
\mathbb{C}
$), the fractional derivative of order $\rho $ is stated by%
\begin{equation*}
D_{z}^{\rho }f(z)=\frac{1}{\Gamma (1-\rho )}\frac{d}{dz}\int\limits_{0}^{z}%
\frac{f(\xi )}{(z-\xi )^{\rho }}d\xi \ \ \ (0\leq \rho <1)
\end{equation*}%
and the fractional integral of order $\rho $ is stated by%
\begin{equation*}
I_{z}^{\rho }f(z)=\frac{1}{\Gamma (\rho )}\int\limits_{0}^{z}f(\xi )(z-\xi
)^{\rho -1}d\xi \ \ \ (\rho >0).
\end{equation*}
\end{definition}

\begin{definition}
(See \cite{S}) The Tremblay fractional derivative operator of the function $%
f $ is defined as%
\begin{equation}
I_{z}^{\mu ,\rho }f(z)=\frac{\Gamma (\rho )}{\Gamma (\mu )}z^{1-\rho
}D_{z}^{\mu -\rho }z^{\mu -1}f(z)\ \ \ (0<\mu \leq 1,0<\rho \leq 1,\mu \geq
\rho ,0<\mu -\rho <1).  \label{eq3}
\end{equation}
\end{definition}

From (\ref{eq3}), we deduce that%
\begin{equation*}
I_{z}^{\mu ,\rho }f(z)=\frac{\mu }{\rho }z+\overset{\infty }{\underset{n=2}{%
\sum }}\frac{\Gamma (\rho )\Gamma (n+\mu )}{\Gamma (\mu )\Gamma (n+\rho )}%
a_{n}z^{n}.~
\end{equation*}

In this paper, we study the new class $\mathfrak{R}_{\Sigma ,\gamma }^{\mu
,\rho }\left( \widetilde{\mathfrak{p}}\right) $ of bi-univalent functions
established by using the Tremblay fractional derivative operator. Further,
we use the Faber polynomial expansions and Fibonacci numbers to derive
bounds for the general coefficient $\left\vert a_{n}\right\vert $ of the
bi-univalent function class.

\section{Preliminaries}

By utilizing the Faber polynomial expansions for functions $f$ $\in \mathcal{%
A}$ of the form (\ref{eq1}), the coefficients of its inverse map $g=f$ $^{-1}
$ may be stated by  \cite{Airault and Bouali 2006, Airault and Ren 2002}:

\begin{equation*}
g\left( w\right) =f^{-1}\left( w\right) =w+\overset{\infty }{\underset{n=2}{%
\sum }}\frac{1}{n}K_{n-1}^{-n}\left( a_{2},a_{3},...\right) w^{n},
\end{equation*}%
where

\begin{eqnarray*}
K_{n-1}^{-n} &=&\frac{\left( -n\right) !}{\left( -2n+1\right) !\left(
n-1\right) !}a_{2}^{n-1}~+\frac{\left( -n\right) !}{\left[ 2\left(
-n+1\right) \right] !\left( n-3\right) !}a_{2}^{n-3}a_{3}~ \\
&&+~\frac{\left( -n\right) !}{\left( -2n+3\right) !\left( n-4\right) !}%
a_{2}^{n-4}a_{4}~ \\
&&+\frac{\left( -n\right) !}{\left[ 2\left( -n+2\right) \right] !\left(
n-5\right) !}a_{2}^{n-5}\left[ a_{5}+\left( -n+2\right) a_{3}^{2}\right] \\
&&+\frac{\left( -n\right) !}{\left( -2n+5\right) !\left( n-6\right) !}%
a_{2}^{n-6}\left[ a_{6}+\left( -2n+5\right) a_{3}a_{4}\right] \\
&&+\overset{}{\underset{j\geq 7}{\sum }}a_{2}^{n-j}V_{j},
\end{eqnarray*}%
such that $V_{j}$ $\left( 7\leq j\leq n\right) $ is a homogeneous polynomial
in the variables $a_{2},a_{3},...,a_{n}$. In the following, the first three
terms of $K_{n-1}^{-n}$ are stated by

\begin{eqnarray*}
\frac{1}{2}K_{1}^{-2} &=&-a_{2}, \\
\frac{1}{3}K_{2}^{-3} &=&2a_{2}^{2}-a_{3}, \\
\frac{1}{4}K_{3}^{-4} &=&-\left( 5a_{2}^{3}-5a_{2}a_{3}+a_{4}\right) .
\end{eqnarray*}%
In general, the expansion of $K_{n}^{p}$ $(p\in 
\mathbb{Z}
=\left\{ 0,\pm 1,\pm 2,\ldots \right\} )$ is stated by

\begin{equation*}
K_{n}^{p}=pa_{n}+\frac{p\left( p-1\right) }{2}\mathcal{G}_{n}^{2}+\frac{p!}{%
\left( p-3\right) !3!}\mathcal{G}_{n}^{3}+...+\frac{p!}{\left( p-n\right) !n!%
}\mathcal{G}_{n}^{n},
\end{equation*}%
where $\mathcal{G}_{n}^{p}=$ $\mathcal{G}_{n}^{p}\left(
a_{1},a_{2},...\right) $ and by \cite{Airault 2007},

\begin{equation*}
\mathcal{G}_{n}^{m}\left( a_{1},a_{2},...,a_{n}\right) =\overset{\infty }{%
\underset{n=1}{\sum }}\frac{m!\left( a_{1}\right) ^{\delta _{1}}...\left(
a_{n}\right) ^{\delta _{n}}}{\delta _{1}!...\delta _{n}!},
\end{equation*}%
while $a_{1}=1$, the sum is taken over all nonnegative integers $\delta
_{1},...,\delta _{n}$ satisfying%
\begin{eqnarray*}
\delta _{1}+\delta _{2}+~...~+\delta _{n} &=&m, \\
\delta _{1}+2\delta _{2}+~...~+n\delta _{n} &=&n.
\end{eqnarray*}%
The first and the last polynomials are%
\begin{equation*}
\mathcal{G}_{n}^{1}=a_{n}\ \ \ \ \ \ \ \ \mathcal{G}_{n}^{n}=a_{1}^{n}.
\end{equation*}%
For two analytic functions $\mathfrak{u}\left( z\right) $, $\mathfrak{v}%
\left( w\right) $ $\left( \mathfrak{u}\left( 0\right) =\mathfrak{v}\left(
0\right) =0,\ \left\vert \mathfrak{u}\left( z\right) \right\vert <1,\
\left\vert \mathfrak{v}\left( w\right) \right\vert <1\right) ,\ $suppose that%
\begin{equation*}
\begin{array}{l}
\mathfrak{u}\left( z\right) =\sum_{n=1}^{\infty }t_{n}z^{n}\ \ \left(
\left\vert z\right\vert <1,\ z\in \mathbb{U}\right) \ \ \ , \\ 
\\ 
\mathfrak{v}\left( w\right) =\sum_{n=1}^{\infty }s_{n}w^{n}\ \ \left(
\left\vert w\right\vert <1,\ w\in \mathbb{U}\right) .%
\end{array}%
\end{equation*}%
It is well known that

\begin{equation}
\left\vert t_{1}\right\vert \leq 1,\ \ \left\vert t_{2}\right\vert \leq
1-\left\vert t_{1}\right\vert ^{2},\ \ \left\vert s_{1}\right\vert \leq 1,\
\ \left\vert s_{2}\right\vert \leq 1-\left\vert s_{1}\right\vert ^{2}.
\label{eq9}
\end{equation}

\begin{definition}
A function $f\in \Sigma $ is said to be in the class%
\begin{equation*}
\mathfrak{R}_{\Sigma ,\gamma }^{\mu ,\rho }\left( \widetilde{\mathfrak{p}}%
\right) \ \ \ (\gamma \in 
\mathbb{C}
\backslash \{0\},\ 0<\mu \leq 1,\ 0<\rho \leq 1,\ z,w\in \mathbb{U})
\end{equation*}%
if the following subordination relationships are satisfied:%
\begin{equation*}
\left[ 1+\frac{1}{\gamma }\left( \frac{\rho \left( I_{z}^{\mu ,\rho
}f(z)\right) ^{\prime }}{\mu }-1\right) \right] \prec \widetilde{\mathfrak{p}%
}\left( z\right) =\frac{1+\tau ^{2}z^{2}}{1-\tau z-\tau ^{2}z^{2}}
\end{equation*}%
and%
\begin{equation*}
\left[ 1+\frac{1}{\gamma }\left( \frac{\rho \left( I_{z}^{\mu ,\rho
}g(w)\right) ^{\prime }}{\mu }-1\right) \right] \prec \widetilde{\mathfrak{p}%
}\left( w\right) =\frac{1+\tau ^{2}w^{2}}{1-\tau w-\tau ^{2}w^{2}},
\end{equation*}%
where the function $g$ is given by (\ref{eq2}) and $\tau =\frac{1-\sqrt{5}}{2%
}\approx -0.618.$
\end{definition}

\begin{remark}
The function $\widetilde{\mathfrak{p}}\left( z\right) $ is not univalent in $%
\mathbb{U}$, but it is univalent in the disc $\left\vert z\right\vert <\frac{%
3-\sqrt{5}}{2}\approx -0.38$. For example, $\widetilde{\mathfrak{p}}\left(
0\right) =\widetilde{\mathfrak{p}}\left( -\frac{1}{2\tau }\right) $ and $%
\widetilde{\mathfrak{p}}\left( e^{\pm i\arccos (1/4)}\right) =\frac{\sqrt{5}%
}{5}$. Also, it can be written as%
\begin{equation*}
\frac{1}{\left\vert \tau \right\vert }=\frac{\left\vert \tau \right\vert }{%
1-\left\vert \tau \right\vert }
\end{equation*}%
which indicates that the number $\left\vert \tau \right\vert $ divides $%
\left[ 0,1\right] $ such that it fulfills the golden section (see for
details Dziok et al. \cite{D}).
\end{remark}

Additionally, Dziok et al. \cite{D} indicate a useful connection between the
function $\widetilde{\mathfrak{p}}\left( z\right) $ and the Fibonacci
numbers. Let $\left\{ \Lambda _{n}\right\} $ be the sequence of Fibonacci
numbers 
\begin{equation*}
\Lambda _{0}=0,\ \Lambda _{1}=1,\ \Lambda _{n+2}=\Lambda _{n}+\Lambda
_{n+1}\ (n\in 
\mathbb{N}
_{0}=\left\{ 0,1,2,\ldots \right\} ),
\end{equation*}%
then%
\begin{equation*}
\Lambda _{n}=\frac{(1-\tau )^{n}-\tau ^{n}}{\sqrt{5}},\ \ \tau =\frac{1-%
\sqrt{5}}{2}.
\end{equation*}%
If we set 
\begin{eqnarray*}
\widetilde{\mathfrak{p}}\left( z\right) &=&1+\overset{\infty }{\underset{n=1}%
{\sum }}\widetilde{\mathfrak{p}}_{n}z^{n}=1+(\Lambda _{0}+\Lambda _{2})\tau
z+(\Lambda _{1}+\Lambda _{3})\tau ^{2}z^{2} \\
&& \\
&&+\overset{\infty }{\underset{n=3}{\sum }}(\Lambda _{n-3}+\Lambda
_{n-2}+\Lambda _{n-1}+\Lambda _{n})\tau ^{n}z^{n},
\end{eqnarray*}%
then the coefficients $\widetilde{\mathfrak{p}}_{n}$ satisfy%
\begin{equation}
\widetilde{\mathfrak{p}}_{n}=\left\{ 
\begin{array}{ll}
\tau & \left( n=1\right) \\ 
&  \\ 
3\tau ^{2} & \left( n=2\right) \\ 
&  \\ 
\tau \widetilde{\mathfrak{p}}_{n-1}+\tau ^{2}\widetilde{\mathfrak{p}}_{n-2}
& \left( n=3,4,\ldots \right)%
\end{array}%
\right. .  \label{D}
\end{equation}

Specializing the parameters $\gamma ,\mu $ and $\rho $, we state the
following definitions.

\begin{definition}
For $\mu =\rho =1,$ a function $f\in \Sigma $ is said to be in the class $%
\mathfrak{R}_{\Sigma ,\gamma }\left( \widetilde{\mathfrak{p}}\right) \left(
\gamma \in 
\mathbb{C}
\backslash \{0\}\right) $ if it satisfies the following conditions
respectively:%
\begin{equation*}
\left[ 1+\frac{1}{\gamma }\left( f^{\prime }(z)-1\right) \right] \prec 
\widetilde{\mathfrak{p}}\left( z\right)
\end{equation*}%
and%
\begin{equation*}
\left[ 1+\frac{1}{\gamma }\left( g^{\prime }(w)-1\right) \right] \prec 
\widetilde{\mathfrak{p}}\left( w\right) ,
\end{equation*}%
where $g=f^{-1}.$
\end{definition}

\begin{definition}
For $\gamma =\mu =\rho =1,$ a function $f\in \Sigma $ is said to be in the
class $\mathfrak{R}_{\Sigma }\left( \widetilde{\mathfrak{p}}\right) $ if it
satisfies the following conditions respectively:%
\begin{equation*}
f^{\prime }(z)\prec \widetilde{\mathfrak{p}}\left( z\right)
\end{equation*}%
and%
\begin{equation*}
g^{\prime }(w)\prec \widetilde{\mathfrak{p}}\left( w\right) ,
\end{equation*}%
where $g=f^{-1}.$
\end{definition}

\section{Main Result and its consequences}

\begin{theorem}
For $\gamma \in 
\mathbb{C}
\backslash \{0\}$, let $f\in \mathfrak{R}_{\Sigma ,\gamma }^{\mu ,\rho
}\left( \widetilde{\mathfrak{p}}\right) $. If $a_{m}=0~\left( 2\leq m\leq
n-1\right) $, then 
\begin{equation*}
\left\vert a_{n}\right\vert \leq \frac{\left\vert \gamma \right\vert
\left\vert \tau \right\vert \Gamma (\mu +1)\Gamma (n+\rho )}{n\Gamma (\rho
+1)\Gamma (n+\mu )}\ \ \ (n\geq 3).
\end{equation*}
\end{theorem}

\begin{proof}
Let $f$ be given by (\ref{eq1}). By the definition of subordination yields%
\begin{equation}
\left[ 1+\frac{1}{\gamma }\left( \frac{\rho \left( I_{z}^{\mu ,\rho
}f(z)\right) ^{\prime }}{\mu }-1\right) \right] =\widetilde{\mathfrak{p}}(%
\mathfrak{u}(z))  \label{eq16}
\end{equation}%
and%
\begin{equation}
\left[ 1+\frac{1}{\gamma }\left( \frac{\rho \left( I_{z}^{\mu ,\rho
}g(w)\right) ^{\prime }}{\mu }-1\right) \right] =\widetilde{\mathfrak{p}}(%
\mathfrak{v}(w)).  \label{eq17}
\end{equation}%
Now, an application of Faber polynomial expansion to the power series $%
\mathfrak{R}_{\Sigma ,\gamma }^{\mu ,\rho }\left( \widetilde{\mathfrak{p}}%
\right) $ (e.g. see \cite{Airault and Bouali 2006} or [\cite{Airault and Ren
2002}, equation (1.6)]) yields 
\begin{equation*}
1+\frac{1}{\gamma }\left( \frac{\rho \left( I_{z}^{\mu ,\rho }f(z)\right)
^{\prime }}{\mu }-1\right) =1+\frac{\Gamma (\rho +1)}{\gamma \Gamma (\mu +1)}%
\overset{\infty }{\underset{n=2}{\sum }}\mathcal{F}_{n-1}\left(
a_{2},a_{3},...,a_{n}\right) z^{n-1}
\end{equation*}%
where%
\begin{equation*}
\begin{array}{ll}
\mathcal{F}_{n-1}\left( a_{2},a_{3},...,a_{n}\right) z^{n-1} & =n\frac{%
\Gamma (n+\mu )}{\Gamma (n+\rho )} \\ 
&  \\ 
& \times \overset{\infty }{\underset{i_{1}+2i_{2}+\cdots +(n-1)i_{(n-1)}=n-1}%
{\sum }}\frac{\left( 1-\left( i_{1}+i_{2}+\cdots +i_{n-1}\right) \right) !%
\left[ \left( a_{2}\right) ^{i_{1}}\left( a_{3}\right) ^{i_{2}}...\left(
a_{n}\right) ^{i_{n-1}}\right] }{\left( i_{1}!\right) \left( i_{2}!\right)
...\left( i_{n-1}!\right) }%
\end{array}%
\end{equation*}%
\begin{equation*}
\end{equation*}%
In particular, the first two terms are, $\mathcal{F}_{1}=\frac{2(\mu +1)}{%
\gamma (\rho +1)}a_{2},\mathcal{F}_{1}=\frac{3(\mu +1)(\mu +2)}{\gamma (\rho
+1)(\rho +2)}a_{3}.$

By the same token, for its inverse map $g=f^{-1}$, it is seen that 
\begin{eqnarray*}
1+\frac{1}{\gamma }\left( \frac{\rho \left( I_{z}^{\mu ,\rho }g(w)\right)
^{\prime }}{\mu }-1\right) &=&1+\overset{\infty }{\underset{n=2}{\sum }}%
\frac{\Gamma (\rho +1)\Gamma (n+\mu )}{\Gamma (\mu +1)\Gamma (n+\rho )}\frac{%
n}{\gamma }\times \frac{1}{n}K_{n-1}^{-n}\left( a_{2},a_{3},...\right)
w^{n-1} \\
&& \\
&=&1+\frac{\Gamma (\rho +1)}{\gamma \Gamma (\mu +1)}\overset{\infty }{%
\underset{n=2}{\sum }}\mathcal{F}_{n-1}\left( b_{2},b_{3},...,b_{n}\right)
w^{n-1}.
\end{eqnarray*}%
Next, the equations (\ref{eq16}) and (\ref{eq17}) lead to%
\begin{eqnarray*}
\widetilde{\mathfrak{p}}\left( \mathfrak{u}\left( z\right) \right) &=&1+%
\widetilde{\mathfrak{p}}_{1}\mathfrak{u}(z)+\widetilde{\mathfrak{p}}_{2}(%
\mathfrak{u}(z))^{2}z^{2}+\cdots \\
&& \\
&=&1+\widetilde{\mathfrak{p}}_{1}t_{1}z+\left( \widetilde{\mathfrak{p}}%
_{1}t_{2}+\widetilde{\mathfrak{p}}_{2}t_{1}^{2}\right) z^{2}+\cdots \\
&& \\
&=&1+\underset{}{\underset{n=1}{\overset{\infty }{\sum }}}\underset{k=1}{%
\overset{n}{\sum }}\widetilde{\mathfrak{p}}_{k}\mathcal{G}_{n}^{k}\left(
t_{1},t_{2},...,t_{n}\right) z^{n},
\end{eqnarray*}%
and 
\begin{eqnarray*}
\widetilde{\mathfrak{p}}\left( \mathfrak{v}\left( w\right) \right) &=&1+%
\widetilde{\mathfrak{p}}_{1}\mathfrak{v}(w)+\widetilde{\mathfrak{p}}_{2}(%
\mathfrak{v}(w))^{2}z^{2}+\cdots \\
&& \\
&=&1+\widetilde{\mathfrak{p}}_{1}s_{1}w+\left( \widetilde{\mathfrak{p}}%
_{1}s_{2}+\widetilde{\mathfrak{p}}_{2}s_{1}^{2}\right) w^{2}+\cdots \\
&& \\
&=&1+\underset{}{\underset{n=1}{\overset{\infty }{\sum }}}\underset{k=1}{%
\overset{n}{\sum }}\widetilde{\mathfrak{p}}_{k}\mathcal{G}_{n}^{k}\left(
s_{1},s_{2},...,s_{n}\right) w^{n}.
\end{eqnarray*}%
Comparing the corresponding coefficients of (\ref{eq16}) and (\ref{eq17})
yields%
\begin{equation*}
\frac{\Gamma (\rho +1)\Gamma (n+\mu )}{\Gamma (\mu +1)\Gamma (n+\rho )}\frac{%
n}{\gamma }a_{n}=\widetilde{\mathfrak{p}}_{1}t_{n-1,}
\end{equation*}%
\begin{equation*}
\frac{\Gamma (\rho +1)\Gamma (n+\mu )}{\Gamma (\mu +1)\Gamma (n+\rho )}\frac{%
n}{\gamma }b_{n}=\widetilde{\mathfrak{p}}_{1}s_{n-1}.
\end{equation*}%
For $a_{m}=0\ \left( 2\leq m\leq n-1\right) ,$ we get $b_{n}=-a_{n}$ and so%
\begin{equation}
\frac{\Gamma (\rho +1)\Gamma (n+\mu )}{\Gamma (\mu +1)\Gamma (n+\rho )}\frac{%
n}{\gamma }a_{n}=\widetilde{\mathfrak{p}}_{1}t_{n-1}  \label{eq18}
\end{equation}%
and%
\begin{equation}
-\frac{\Gamma (\rho +1)\Gamma (n+\mu )}{\Gamma (\mu +1)\Gamma (n+\rho )}%
\frac{n}{\gamma }a_{n}=\widetilde{\mathfrak{p}}_{1}s_{n-1}.  \label{eq19}
\end{equation}%
Now taking the absolute values of either of the above two equations and from
(\ref{eq9}), we obtain%
\begin{equation*}
\left\vert a_{n}\right\vert \leq \frac{\left\vert \gamma \right\vert
\left\vert \tau \right\vert \Gamma (\mu +1)\Gamma (n+\rho )}{n\Gamma (\rho
+1)\Gamma (n+\mu )}.
\end{equation*}
\end{proof}

\begin{corollary}
For $\gamma \in 
\mathbb{C}
\backslash \{0\}$, suppose that $f\in \mathfrak{R}_{\Sigma ,\gamma }\left( 
\widetilde{\mathfrak{p}}\right) $. If $a_{m}=0~\left( 2\leq m\leq n-1\right) 
$, then 
\begin{equation*}
\left\vert a_{n}\right\vert \leq \frac{\left\vert \gamma \right\vert
\left\vert \tau \right\vert }{n}\ \ \ (n\geq 3).
\end{equation*}
\end{corollary}

\begin{corollary}
Suppose that $f\in \mathfrak{R}_{\Sigma }\left( \widetilde{\mathfrak{p}}%
\right) $. If $a_{m}=0~\left( 2\leq m\leq n-1\right) $, then 
\begin{equation*}
\left\vert a_{n}\right\vert \leq \frac{\left\vert \tau \right\vert }{n}\ \ \
(n\geq 3).
\end{equation*}
\end{corollary}

\begin{theorem}
Let $f\in \mathfrak{R}_{\Sigma ,\gamma }^{\mu ,\rho }\left( \widetilde{%
\mathfrak{p}}\right) \ (\gamma \in 
\mathbb{C}
\backslash \{0\}).$Then%
\begin{eqnarray*}
\left\vert a_{2}\right\vert &\leq &\min \left\{ \dfrac{\left\vert \gamma
\right\vert \left\vert \tau \right\vert }{\sqrt{\left\vert \tfrac{3\gamma
(\mu +1)(\mu +2)}{(\rho +1)(\rho +2)}-\tfrac{12(\mu +1)^{2}}{(\rho +1)^{2}}%
\right\vert \left\vert \tau \right\vert +\tfrac{4(\mu +1)^{2}}{(\rho +1)^{2}}%
}},\right. \\
&& \\
&&\left. \left\vert \tau \right\vert \sqrt{\frac{\left\vert \gamma
\right\vert (\rho +1)(\rho +2)}{(\mu +1)(\mu +2)}}\right\}
\end{eqnarray*}%
and%
\begin{eqnarray*}
\left\vert a_{3}\right\vert &\leq &\min \left\{ \frac{\left\vert \gamma
\right\vert \tau ^{2}(\rho +1)(\rho +2)}{(\mu +1)(\mu +2)},\right. \\
&& \\
&&\left. \dfrac{\left\vert \tau \right\vert }{\frac{3(\mu +1)(\mu +2)}{%
\left\vert \gamma \right\vert (\rho +1)(\rho +2)}}\left[ 1+\frac{\left[ 
\frac{3(\mu +1)(\mu +2)\left\vert \gamma \right\vert \left\vert \tau
\right\vert }{(\rho +1)(\rho +2)}-\frac{4(\mu +1)^{2}}{(\rho +1)^{2}}\right] 
}{\left\vert \dfrac{3\gamma (\mu +1)(\mu +2)}{(\rho +1)(\rho +2)}-\dfrac{%
12(\mu +1)^{2}}{(\rho +1)^{2}}\right\vert \left\vert \tau \right\vert +%
\dfrac{4(\mu +1)^{2}}{(\rho +1)^{2}}}\right] \right\} .
\end{eqnarray*}
\end{theorem}

\begin{proof}
Substituting $n$ by $2$ and $3$ in (\ref{eq18}) and (\ref{eq19}),
respectively, we find that%
\begin{equation}
\frac{2(\mu +1)}{\gamma (\rho +1)}a_{2}=\widetilde{\mathfrak{p}}_{1}t_{1},
\label{eq20}
\end{equation}%
\begin{equation}
\frac{3(\mu +1)(\mu +2)}{\gamma (\rho +1)(\rho +2)}a_{3}=\widetilde{%
\mathfrak{p}}_{1}t_{2}+\widetilde{\mathfrak{p}}_{2}t_{1}^{2},  \label{eq21}
\end{equation}%
\begin{equation}
-\frac{2(\mu +1)}{\gamma (\rho +1)}a_{2}=\widetilde{\mathfrak{p}}_{1}s_{1},
\label{eq22}
\end{equation}%
\begin{equation}
\frac{3(\mu +1)(\mu +2)}{\gamma (\rho +1)(\rho +2)}(2a_{2}^{2}-a_{3})=%
\widetilde{\mathfrak{p}}_{1}s_{2}+\widetilde{\mathfrak{p}}_{2}s_{1}^{2}.
\label{eq23}
\end{equation}%
Obviously, we obtain%
\begin{equation}
t_{1}=-s_{1}.  \label{eq24}
\end{equation}%
If we add the equation (\ref{eq23}) to (\ref{eq21}) and use (\ref{eq24}), we
get 
\begin{equation}
\frac{6(\mu +1)(\mu +2)}{\gamma (\rho +1)(\rho +2)}a_{2}^{2}=\widetilde{%
\mathfrak{p}}_{1}\left( t_{2}+s_{2}\right) +2\widetilde{\mathfrak{p}}%
_{2}t_{1}^{2}.  \label{eq25}
\end{equation}%
Using the value of $t_{1}^{2}$ from (\ref{eq20}), we get 
\begin{equation}
\left[ \frac{6(\mu +1)(\mu +2)}{\gamma (\rho +1)(\rho +2)}\widetilde{%
\mathfrak{p}}_{1}^{2}-\frac{8(\mu +1)^{2}}{\gamma ^{2}(\rho +1)^{2}}%
\widetilde{\mathfrak{p}}_{2}\right] a_{2}^{2}=\widetilde{\mathfrak{p}}%
_{1}^{3}\left( t_{2}+s_{2}\right) .  \label{eq26}
\end{equation}%
Combining (\ref{eq26}) and (\ref{eq9}), we obtain 
\begin{eqnarray*}
2\left\vert \frac{3(\mu +1)(\mu +2)}{\gamma (\rho +1)(\rho +2)}\widetilde{%
\mathfrak{p}}_{1}^{2}-\frac{4(\mu +1)^{2}}{\gamma ^{2}(\rho +1)^{2}}%
\widetilde{\mathfrak{p}}_{2}\right\vert \left\vert a_{2}\right\vert ^{2}
&\leq &\left\vert \widetilde{\mathfrak{p}}_{1}\right\vert ^{3}\left(
\left\vert t_{2}\right\vert +\left\vert s_{2}\right\vert \right) \\
&& \\
&\leq &2\left\vert \widetilde{\mathfrak{p}}_{1}\right\vert ^{3}\left(
1-\left\vert t_{1}\right\vert ^{2}\right) \\
&& \\
&=&2\left\vert \widetilde{\mathfrak{p}}_{1}\right\vert ^{3}-2\left\vert 
\widetilde{\mathfrak{p}}_{1}\right\vert ^{3}\left\vert t_{1}\right\vert ^{2}.
\end{eqnarray*}%
It follows from (\ref{eq20}) that%
\begin{equation}
\left\vert a_{2}\right\vert \leq \dfrac{\left\vert \gamma \right\vert
\left\vert \tau \right\vert }{\sqrt{\left\vert \dfrac{3\gamma (\mu +1)(\mu
+2)}{(\rho +1)(\rho +2)}-\dfrac{12(\mu +1)^{2}}{(\rho +1)^{2}}\right\vert
\left\vert \tau \right\vert +\dfrac{4(\mu +1)^{2}}{(\rho +1)^{2}}}}.
\label{eq28}
\end{equation}%
Additionally, by (\ref{eq9}) and (\ref{eq25}) 
\begin{eqnarray*}
\frac{6(\mu +1)(\mu +2)}{\left\vert \gamma \right\vert (\rho +1)(\rho +2)}%
\left\vert a_{2}\right\vert ^{2} &\leq &\left\vert \widetilde{\mathfrak{p}}%
_{1}\right\vert \left( \left\vert t_{2}\right\vert +\left\vert
s_{2}\right\vert \right) +2\left\vert \widetilde{\mathfrak{p}}%
_{2}\right\vert \left\vert t_{1}\right\vert ^{2} \\
&& \\
&\leq &2\left\vert \widetilde{\mathfrak{p}}_{1}\right\vert \left(
1-\left\vert t_{1}\right\vert ^{2}\right) +2\left\vert \widetilde{\mathfrak{p%
}}_{2}\right\vert \left\vert t_{1}\right\vert ^{2} \\
&& \\
&=&2\left\vert \widetilde{\mathfrak{p}}_{1}\right\vert +2\left\vert
t_{1}\right\vert ^{2}(\left\vert \widetilde{\mathfrak{p}}_{2}\right\vert
-\left\vert \widetilde{\mathfrak{p}}_{1}\right\vert ).
\end{eqnarray*}%
Since $\left\vert \widetilde{\mathfrak{p}}_{2}\right\vert >\left\vert 
\widetilde{\mathfrak{p}}_{1}\right\vert $, we get%
\begin{equation*}
\left\vert a_{2}\right\vert \leq \left\vert \tau \right\vert \sqrt{\frac{%
\left\vert \gamma \right\vert (\rho +1)(\rho +2)}{(\mu +1)(\mu +2)}}.
\end{equation*}%
Next, in order to derive the bounds on $\left\vert a_{3}\right\vert ,$ by
subtracting (\ref{eq23}) from (\ref{eq21}), we may obtain%
\begin{equation}
\frac{6(\mu +1)(\mu +2)}{\gamma (\rho +1)(\rho +2)}a_{3}=\frac{6(\mu +1)(\mu
+2)}{\gamma (\rho +1)(\rho +2)}a_{2}^{2}+\widetilde{\mathfrak{p}}_{1}\left(
t_{2}-s_{2}\right) .  \label{eq29}
\end{equation}%
Evidently, from (\ref{eq25}), we state that%
\begin{eqnarray*}
a_{3} &=&\frac{\widetilde{\mathfrak{p}}_{1}\left( t_{2}+s_{2}\right) +2%
\widetilde{\mathfrak{p}}_{2}t_{1}^{2}}{\frac{6(\mu +1)(\mu +2)}{\gamma (\rho
+1)(\rho +2)}}+\frac{\widetilde{\mathfrak{p}}_{1}\left( t_{2}-s_{2}\right) }{%
\frac{6(\mu +1)(\mu +2)}{\gamma (\rho +1)(\rho +2)}} \\
&& \\
&=&\frac{\widetilde{\mathfrak{p}}_{1}t_{2}+\widetilde{\mathfrak{p}}%
_{2}t_{1}^{2}}{\frac{3(\mu +1)(\mu +2)}{\gamma (\rho +1)(\rho +2)}}
\end{eqnarray*}%
and consequently%
\begin{eqnarray*}
\left\vert a_{3}\right\vert &\leq &\frac{\left\vert \widetilde{\mathfrak{p}}%
_{1}\right\vert \left\vert t_{2}\right\vert +\left\vert \widetilde{\mathfrak{%
p}}_{2}\right\vert \left\vert t_{1}\right\vert ^{2}}{\frac{3(\mu +1)(\mu +2)%
}{\left\vert \gamma \right\vert (\rho +1)(\rho +2)}} \\
&& \\
&\leq &\frac{\left\vert \widetilde{\mathfrak{p}}_{1}\right\vert \left(
1-\left\vert t_{1}\right\vert ^{2}\right) +\left\vert \widetilde{\mathfrak{p}%
}_{2}\right\vert \left\vert t_{1}\right\vert ^{2}}{\frac{3(\mu +1)(\mu +2)}{%
\left\vert \gamma \right\vert (\rho +1)(\rho +2)}} \\
&& \\
&=&\frac{\left\vert \widetilde{\mathfrak{p}}_{1}\right\vert +\left\vert
t_{1}\right\vert ^{2}(\left\vert \widetilde{\mathfrak{p}}_{2}\right\vert
-\left\vert \widetilde{\mathfrak{p}}_{1}\right\vert )}{\frac{3(\mu +1)(\mu
+2)}{\left\vert \gamma \right\vert (\rho +1)(\rho +2)}}.
\end{eqnarray*}%
Since $\left\vert \widetilde{\mathfrak{p}}_{2}\right\vert >\left\vert 
\widetilde{\mathfrak{p}}_{1}\right\vert $, we must write%
\begin{equation*}
\left\vert a_{3}\right\vert \leq \frac{\left\vert \gamma \right\vert \tau
^{2}(\rho +1)(\rho +2)}{(\mu +1)(\mu +2)}.
\end{equation*}%
On the other hand, by (\ref{eq9}) and (\ref{eq29}), we have 
\begin{eqnarray*}
\frac{6(\mu +1)(\mu +2)}{\left\vert \gamma \right\vert (\rho +1)(\rho +2)}%
\left\vert a_{3}\right\vert &\leq &\frac{6(\mu +1)(\mu +2)}{\left\vert
\gamma \right\vert (\rho +1)(\rho +2)}\left\vert a_{2}\right\vert
^{2}+\left\vert \widetilde{\mathfrak{p}}_{1}\right\vert \left( \left\vert
t_{2}\right\vert +\left\vert s_{2}\right\vert \right) \\
&& \\
&\leq &\frac{6(\mu +1)(\mu +2)}{\left\vert \gamma \right\vert (\rho +1)(\rho
+2)}\left\vert a_{2}\right\vert ^{2}+2\left\vert \widetilde{\mathfrak{p}}%
_{1}\right\vert \left( 1-\left\vert t_{1}\right\vert ^{2}\right) .
\end{eqnarray*}%
Then, with the help of (\ref{eq20}), we have%
\begin{equation*}
\frac{3(\mu +1)(\mu +2)}{\left\vert \gamma \right\vert (\rho +1)(\rho +2)}%
\left\vert a_{3}\right\vert \leq \left[ \frac{3(\mu +1)(\mu +2)}{\left\vert
\gamma \right\vert (\rho +1)(\rho +2)}-\frac{4(\mu +1)^{2}}{\left\vert
\gamma \right\vert ^{2}(\rho +1)^{2}\left\vert \widetilde{\mathfrak{p}}%
_{1}\right\vert }\right] \left\vert a_{2}\right\vert ^{2}+\left\vert 
\widetilde{\mathfrak{p}}_{1}\right\vert .
\end{equation*}%
By considering (\ref{eq28}), we deduce that%
\begin{equation*}
\left\vert a_{3}\right\vert \leq \dfrac{\left\vert \tau \right\vert }{\frac{%
3(\mu +1)(\mu +2)}{\left\vert \gamma \right\vert (\rho +1)(\rho +2)}}\left\{
1+\frac{\left[ \frac{3(\mu +1)(\mu +2)\left\vert \gamma \right\vert
\left\vert \tau \right\vert }{(\rho +1)(\rho +2)}-\frac{4(\mu +1)^{2}}{(\rho
+1)^{2}}\right] }{\left\vert \dfrac{3\gamma (\mu +1)(\mu +2)}{(\rho +1)(\rho
+2)}-\dfrac{12(\mu +1)^{2}}{(\rho +1)^{2}}\right\vert \left\vert \tau
\right\vert +\dfrac{4(\mu +1)^{2}}{(\rho +1)^{2}}}\right\} .
\end{equation*}
\end{proof}

\begin{corollary}
Let $f\in \mathfrak{R}_{\Sigma ,\gamma }\left( \widetilde{\mathfrak{p}}%
\right) \ (\gamma \in 
\mathbb{C}
\backslash \{0\}).$Then%
\begin{equation*}
\left\vert a_{2}\right\vert \leq \min \left\{ \dfrac{\left\vert \gamma
\right\vert \left\vert \tau \right\vert }{\sqrt{3\left\vert \gamma
-4\right\vert \left\vert \tau \right\vert +4}},\left\vert \tau \right\vert 
\sqrt{\left\vert \gamma \right\vert }\right\}
\end{equation*}%
and%
\begin{equation*}
\left\vert a_{3}\right\vert \leq \min \left\{ \left\vert \gamma \right\vert
\left\vert \tau \right\vert ^{2},\dfrac{\left( \left\vert \gamma
-4\right\vert +\left\vert \gamma \right\vert \right) \left\vert \tau
\right\vert ^{2}\left\vert \gamma \right\vert }{3\left\vert \gamma
-4\right\vert \left\vert \tau \right\vert +4}\right\} .
\end{equation*}
\end{corollary}

\begin{corollary}
Let $f\in \mathfrak{R}_{\Sigma }\left( \widetilde{\mathfrak{p}}\right) .$Then%
\begin{equation*}
\left\vert a_{2}\right\vert \leq \dfrac{\left\vert \tau \right\vert }{\sqrt{%
9\left\vert \tau \right\vert +4}}
\end{equation*}%
and%
\begin{equation*}
\left\vert a_{3}\right\vert \leq \frac{4\left\vert \tau \right\vert ^{2}}{%
9\left\vert \tau \right\vert +4}.
\end{equation*}
\end{corollary}

\end{document}